\documentclass[a4paper]{amsart}
\usepackage{amsmath}
\usepackage{amssymb}
\usepackage{amsfonts}
\usepackage{pxfonts}
\usepackage[OT2,T1]{fontenc}

\usepackage{accents}

\usepackage[
textwidth=14.5cm, 
textheight=21cm,
hmarginratio=1:1,
vmarginratio=1:1]{geometry}

\newtheorem{thm}{Theorem}[section]
\newtheorem{cor}[thm]{Corollary}
\newtheorem{lem}[thm]{Lemma}
\newtheorem{prop}[thm]{Proposition}
\theoremstyle{definition}

\theoremstyle{remark}
\newtheorem{rem}[thm]{Remark}

\numberwithin{equation}{section}
\numberwithin{figure}{section}

\newcommand*{\dottt}[1]{%
   \accentset{\mbox{\large\bfseries .}}{#1}}

\newcommand{\sm}{\setminus}
\newcommand{\diff}{\mathrm{d}}
\newcommand{\C}{{\mathbb C}}
\newcommand{\R}{{\mathbb R}}
\newcommand{\dA}{{\diff A}}
\newcommand{\ds}{\diff s}
\newcommand{\normv}{\mathrm{n}}
\newcommand{\imag}{\mathrm{i}}

\newcommand{\Pop}{{\mathbf P}}
\newcommand{\Qop}{{\mathbf Q}}
\newcommand{\Gop}{{\mathbf G}}

\newcommand{\Ex}{\mathbb{E}}

\newcommand{\abs}[1]{\lvert#1\rvert}

\newcommand{\norm}[1]{\lVert#1\rVert}
\newcommand{\ip}[1]{\langle#1\rangle}

\newcommand{\cinterv}[1]{\mathopen[{#1}\mathclose]}

\DeclareMathOperator{\cl}{cl}

\begin{document}

%
\title{The Gaussian free field and Hadamard's variational formula}

\author{Haakan~Hedenmalm}
\address{Hedenmalm: Department of Mathematics\\
KTH Royal Institute of Technology\\
S--10044 Stockholm\\
Sweden}

\email{haakanh@math.kth.se}

\author{Pekka~J.~Nieminen}

\address{Nieminen: Department of Mathematics and Statistics\\
University of Helsinki\\
Box 68\\
FI--00014 Helsinki\\
Finland}

\email{pjniemin@cc.helsinki.fi}

\date{8 February 2012}


\thanks{Both authors were supported by the G\"oran Gustafsson
Foundation (KVA) and Vetenskapsr\r{a}det (VR). The second author was
in addition supported by the Academy of Finland, project 136785.}

\begin{abstract} 
We relate the Gaussian free field on a planar domain to the
variational formula of Hadamard which explains the change of the Green
function under a perturbation of the domain.
This is accomplished by means of a natural integral operator -- 
called the Hadamard operator -- associated with a given flow of
growing domains.
\end{abstract}

\maketitle

\section{Introduction} 

The Gaussian free field ($\mathrm{GFF}_0$) on a planar domain
$\Omega$ is an analogue of the classical Brownian motion with two temporal
dimensions, which frequently appears as the continuum limit of discrete random
processes and plays an important role in statistical physics (see \cite{Sh} 
for a general survey). It can be defined, at least formally, as a random
linear combination of an orthonormal basis $\psi_j$ ($j=1,2,3,\ldots$) of the
Sobolev space $H^1_0(\Omega)$,
\[
   \Psi = \sum_{j=1}^{+\infty} \xi_j\psi_j,
\]
where $\xi_j$ ($j=1,2,3,\ldots$) is a sequence of independent standard 
real-valued  Gaussian variables (i.e., taken from $\mathrm{N}(0,1)$). 
The space $H^1_0(\Omega)$ consists of real-valued functions and is
equipped with the usual Dirichlet inner product.
The resulting stochastic field $\Psi$ is conformally invariant with covariance
structure determined by the Green function of the domain; thus
the $\mathrm{GFF}_0$ also exhibits interesting connections to
complex analysis and potential theory.

The purpose of this note is to relate the Gaussian free field to the
classical variational formula due to Hadamard \cite{H} which describes
the change of the Green function under a perturbation of the boundary
of the underlying domain. We interpret Hadamard's formula in
terms of a natural and concrete integral operator -- called by us the
\emph{Hadamard operator} --
associated with a given flow of smooth domains. It supplies an isometric
isomorphism $L^2(\Omega)\to H^1_0(\Omega)$, and, consequently, it can be 
used to produce the corresponding $\mathrm{GFF}_0$ field from the standard 
white noise field. This
construction has certain appealing features since it enables us to
view $\mathrm{GFF}_0$ as formed by adding independent
infinitesimal harmonic fields whose boundary amplitudes correspond
to suitably weighted white noise processes.

The paper is organized as follows. In Section~\ref{sec:Pre}, we
recall the definition of the space $H_0^1(\Omega)$ and some
related prerequisites. In Section~\ref{sec:Had}, we first describe
a version of Hadamard's variational formula, and then proceed
to introduce the associated Hadamard operator and analyse its
properties. Section~\ref{sec:Random} contains a brief account
on the white noise and Gaussian free fields from the point of
view of Gaussian Hilbert spaces. Finally, in
Section~\ref{sec:GFF}, we use the Hadamard operator to construct
the Gaussian free field and derive some natural consequences.

\section{Preliminaries}
\label{sec:Pre}

In this section, $\Omega$ is a finitely connected bounded domain in the 
complex plane $\C$ with a $C^2$-smooth boundary. We use the notation
\[
\Delta:=\frac{\partial^2}{\partial x^2} + \frac{\partial^2}{\partial y^2},
\qquad
\nabla:=\biggl(\frac{\partial}{\partial x},\frac{\partial}{\partial y}\biggr),
\qquad
\dA(z):=\diff x\diff y,
\]
for the Laplacian, the gradient (nabla) operator, and the area element, 
respectively; here, $z=x+\imag y$ is the decomposition of $z\in\C$ into real 
and imaginary parts.

\subsection{The Green function and Poisson kernel}
\label{sec:Green}
The Dirichlet Green function $G_\Omega$ in $\Omega$ solves, for $w\in\Omega$,
\[
\begin{cases}
   -\Delta_z G_\Omega(z,w)=\delta_w(z), & \quad z\in\Omega, \\
   G_\Omega(z,w) = 0,  & \quad z\in\partial\Omega,
\end{cases}
\]
where $\Delta_z$ indicates that the Laplacian is taken with respect to $z$.
It has the symmetry property $G_\Omega(z,w)=G_\Omega(w,z)$.
We write $\Gop_\Omega$ for the associated integral operator
\[
   \Gop_\Omega f(z):=\int_\Omega G_\Omega(z,w)f(w)\,\dA(w),
   \qquad z\in\Omega,
\]
extended to vanish in $\C\sm\Omega$. Then, in the sense
of distributions, $-\Delta \Gop_\Omega f = f$ in $\Omega$, so we may
write $\Gop_\Omega = [-\Delta]^{-1}$.

The Poisson kernel $P_\Omega$ is obtained from 
\[
   P_\Omega(z,\zeta) :=
   \frac{\partial}{\partial\normv(\zeta)}G_\Omega(z,\zeta),
   \qquad z\in\Omega,\ \zeta\in\partial\Omega,
\] 
where the unit normal derivative is taken in the inward direction. For 
$\zeta\in\partial\Omega$, we extend $P_\Omega(z,\zeta)$ to (area-a.e.)
$z\in\C$ 
by setting it equal to $0$ in $\C\sm\cl[\Omega]$, where $\cl[\Omega]$
denotes the closure of $\Omega$. The 
associated integral operator is
\[
   \Pop_\Omega f(z) :=
   \int_{\partial\Omega} P_\Omega(z,\zeta)f(\zeta)\,\ds(\zeta),
   \qquad z\in\Omega,
\]
where $\ds$ refers to the arc-length measure. It furnishes the
harmonic extension of $f$ defined on $\partial\Omega$ to the
interior $\Omega$.  

\subsection{The space $H^1_0(\Omega)$}
\label{sec:Sobolev}
The Sobolev space $H^1_0(\Omega)$, alternatively denoted
by $W^{1,2}_0(\Omega)$, can be defined as the Hilbert space completion
of (real-valued) $C_0^\infty(\Omega)$ under the Dirichlet inner
product 
\[
   \ip{f,g}_{\Omega,\nabla} := \int_\Omega \nabla f\cdot\nabla g\,\dA,
\]
where the dot is the inner product in $\R^2$.
The norm of the space is given by $\norm{f}_{\Omega,\nabla} :=
\ip{f,f}_{\Omega,\nabla}^{1/2}$.
As all functions $f,g \in H^1_0(\Omega)$ vanish on the
boundary, we may write, using integration by parts,
\begin{equation} \label{eq:H1ip}
   \ip{f,g}_{\Omega,\nabla} = \ip{-\Delta f,g}_\Omega
   = \ip{[-\Delta]^{1/2}f,[-\Delta]^{1/2}g}_\Omega,
\end{equation}
where $\ip{\cdot,\cdot}_\Omega$ is the usual inner product in
$L^2(\Omega)$. It follows that the operator
$[-\Delta]^{1/2}$ supplies an isometric isomorphism from
$H_0^1(\Omega)$ onto $L^2(\Omega)$, and its inverse is given
by $[-\Delta]^{-1/2}$. As is standard, the powers of
$-\Delta$ are understood in terms of eigenfunction expansions
with vanishing Dirichlet boundary data (cf., e.g.,
\cite[Sec.~6.5.1]{Evans}).

\section{Hadamard's formula and an associated operator}
\label{sec:Had}

\subsection{Hadamard's variational formula} \label{sec:HVar}
We consider a family of bounded planar domains $\Omega(t)$
($0 < t \leq 1$) such that for some positive integer $N$: 
\begin{enumerate}
\item[(a)]
$\cl[\Omega(t)] \subset \Omega(t')$ for
$0 < t < t' \le 1$.
\item[(b)]
The complement $\C\sm\Omega(t)$ consists of $N$ connectivity components
for all $t$.
\item[(c)]
The boundaries $\partial\Omega(t)$ are all $C^2$-smooth and vary in a
$C^2$-smooth fashion with $t$.
\item[(d)]
The intersection $\bigcap_t\Omega(t)$ -- what might be called \emph{the
skeleton} -- is a continuum which is either a point or the union of
a finite number of $C^2$-smooth curves, and, in particular, has zero
area.
\end{enumerate}
To simplify the notation, we write $G_t$ and $P_t$ for
the Green function and the Poisson kernel of $\Omega(t)$,
and $\Gop_t$ and $\Pop_t$ for the associated operators, respectively.

In the setting of (a)--(d), the classical variational formula of
Hadamard~\cite{H} (cf.\ also \cite{Neh,Sch}) tells us how
the Green function $G_t$ varies with $t$. We prefer to state the
expression in integral form:
\begin{equation}
   G_t(z,w)=\int_0^t\int_{\partial\Omega(\tau)}P_\tau(z,\zeta)
   P_\tau(w,\zeta)\varrho(\zeta)\,\ds(\zeta)\,\diff\tau,
\label{eq:H1}
\end{equation}
where $\varrho(\zeta)$ is the rate at which the boundary 
$\partial\Omega(\tau)$ moves at $\zeta\in\partial\Omega(\tau)$ along
the direction of the exterior unit normal vector as $\tau$ grows. Note
that the index $\tau$ is uniquely determined by the condition
$\zeta\in\partial\Omega(\tau)$. Also recall that the Poisson kernel is
extended to equal $0$ outside
the closure of the domain where it was originally defined.

For our purposes it is important to note that, for each $t$, the
boundaries $\partial\Omega(\tau)$ ($0 < \tau < t$) essentially foliate
the entire domain $\Omega(t)$ in the sense that their union (which is
disjoint by assumption (a)) equals $\Omega(t)$
except for the skeleton. This makes it possible to compute area
integrals over $\Omega(t)$ using the expression
\begin{equation} 
\label{eq:polcoord}
   \int_{\Omega(t)} f \,\dA =
   \int_0^t \biggl\{ \int_{\partial\Omega(\tau)}
   f\varrho\,\ds \biggr\} \,\diff\tau,
\end{equation}
reminiscent of polar coordinates. It follows from \eqref{eq:polcoord}
that $\varrho > 0$ almost everywhere on $\Omega(1)$.

\begin{rem}
An explicit description of a setup of the above kind and a proof of
Hadamard's formula in that context can be found in Schippers~\cite{Sch}.
He assumes that the boundaries $\partial\Omega(t)$ are furnished
by a collection of injective and non-overlapping $C^2$-smooth
homotopies $F_j(t,\theta)$ ($j=1,\ldots,N$) where
$\theta \in \cinterv{0,2\pi}$ (with endpoints identified).
Then, at each point $\zeta = F_j(t,\theta)$, a routine computation
shows that
$\varrho(\zeta) = \abs{\det \mathrm{D}F_j(t,\theta)}/
\abs{\partial F_j(t,\theta)/\partial\theta}$ with
$\mathrm{D}F_j$ denoting the Jacobian matrix. The expression
\eqref{eq:polcoord} now follows by usual change of variables using
the maps $F_j$.
\end{rem}

\subsection{The Hadamard operator}
We keep the setting of (a)--(d) above, and for each
$\zeta \in \Omega(1)$, except the points of the skeleton,
let $\tau(\zeta)$ be the unique value of $t$ 
such that $\zeta\in\partial\Omega(t)$. In words, $\tau(\zeta)$ is the 
time when the front $\partial\Omega(t)$ passes the point $\zeta$. 
For each $0 < t \leq 1$, we introduce the {\em Hadamard operator}
$\Qop_t$ by
\[
   \Qop_t f(z) := \int_{\Omega(t)}P_{\tau(\zeta)}(z,\zeta)
   f(\zeta)\,\dA(\zeta),
\]
whenever the integral makes sense. In terms of signals, $\Qop_t$ maps
the unit point mass at $\zeta\in\Omega(t)$ to the Poisson kernel
function 
$z\mapsto P_{\tau(\zeta)}(z,\zeta)$, treated as zero for 
$z\in\C\sm\Omega(\tau(\zeta))$. Note that $\Qop_t f$ is always
supported in $\cl[\Omega(t)]$ and $\Qop_t f = \Qop_1(1_{\Omega(t)} f)$
where $1_{\Omega(t)}$ stands for the characteristic function of
$\Omega(t)$.
Similarly, we define the {\em adjoint Hadamard operator}
\[
   \Qop_t^* f(\zeta) := 1_{\Omega(t)}(\zeta)
   \int_{\Omega(t)}P_{\tau(\zeta)}(z,\zeta) f(z)\,\dA(z),
\]
whenever the integral makes sense.

We note the linear appearance of the
Poisson kernel in the definition of the Hadamard operator, while it appears 
in a bilinear fashion in Hadamard's formula \eqref{eq:H1}. These two
objects are intimately related; in fact, as we will see in 
Section~\ref{sec:GFF}, Hadamard's formula corresponds to the calculation
of correlations for the stochastic field generated by the Hadamard
operator.

We collect the main properties of $\Qop_t$ and $\Qop_t^*$ in the
following proposition. We denote by $\Delta_t$ the Laplacian
on $\Omega(t)$, and recall that the powers of $-\Delta_t$ are to be
understood with respect to vanishing Dirichlet boundary data.

\begin{prop} \label{prop:uni}
The operators $[-\Delta_t]^{1/2}\Qop_t$ and
$\Qop_t^*[-\Delta_t]^{1/2}$ are both unitary on $L^2(\Omega(t))$.
Moreover, $\Qop_t \Qop_{t'}^* = [-\Delta_{t\wedge t'}]^{-1}$
for every $t$ and $t'$, where $t \wedge t'$ denotes the minimum of
$t$ and~$t'$.
\end{prop}

In view of identity \eqref{eq:H1ip}, the proposition can be rephrased
as saying that $\Qop_t$ is an isometric isomorphism
$L^2(\Omega(t)) \to H^1_0(\Omega(t))$ and $\Qop_t^*$ is an isometric
isomorphism $H^{-1}(\Omega(t)) \to L^2(\Omega(t))$, where
$H^{-1}(\Omega(t))$ is the dual of $H^1_0(\Omega(t))$ under the duality
pairing induced by the $L^2$ inner product.

The following observation, which will be needed in the proof of
Proposition~\ref{prop:uni},
is helpful in understanding the evolution of the operator $\Qop_t$
as the domain $\Omega(t)$ varies.

\begin{lem} \label{le:Qop}
For $0 < t < t' \le 1$ and $f\in L^2(\Omega(t'))$, the 
function 
\[
   \Qop_{t'} f(z)-\Qop_t f(z)
   = \int_{\Omega(t')\sm\Omega(t)}
     P_{\tau(\zeta)}(z,\zeta) f(\zeta)\,\dA(\zeta)
\]
equals $\Qop_{t'} f(z)$ on $\Omega(t')\sm\Omega(t)$, and on
$\Omega(t)$ it is harmonic and its boundary values on $\partial\Omega(t)$ 
are those of $\Qop_{t'} f$.
\end{lem}

\begin{proof}
This follows from the easily verified fact that $\Qop_tf$ vanishes
in $\C\sm\Omega(t)$, together with the observation that 
$z\mapsto P_{\tau(\zeta)}(z,\zeta)$ is harmonic in 
$\Omega(\tau(\zeta))$.
\end{proof}

\begin{proof}[Proof of Proposition~\ref{prop:uni}]
We start by addressing the last assertion. A computation with a
test function $f \in C_0^\infty(\C)$ gives
\[ \begin{split}
   \Qop_t\Qop_{t'}^* f(z)
   &= \int_{\Omega(t')}
     \biggl\{ \int_{\Omega(t)} 1_{\Omega(t')}(\zeta)
     P_{\tau(\zeta)}(z,\zeta) P_{\tau(\zeta)}(w,\zeta)\,\dA(\zeta)
     \biggr\} \,f(w) \,\dA(w)  \\
   &= \int_{\Omega(t')}
     \biggl\{ \int_0^{t\wedge t'} \int_{\partial\Omega(\tau)}
     P_{\tau}(z,\zeta) P_{\tau}(w,\zeta)
     \varrho(\zeta) \,\ds(\zeta)\,\diff\tau \biggr\}
     \,f(w) \,\dA(w),
\end{split} \]
using Fubini's theorem and the ``polar coordinates'' trick
\eqref{eq:polcoord} with $\tau$ as the time variable. However, by
Hadamard's formula \eqref{eq:H1}, the integral in braces equals
$G_{t\wedge t'}(z,w)$. Consequently, $\Qop_t\Qop_{t'}^* f(z) =
\Gop_{t\wedge t'} f(z)$, which gives $\Qop_t\Qop_{t'}^* =\Gop_{t\wedge t'}$;
this is equivalent to the last assertion.

Turning to the operator $\Qop_t^*[-\Delta_t]^{1/2}$, we note that
for all $f,g \in C_0^\infty(\Omega(t))$ we get
\[
   \ip{\Qop_t^*[-\Delta_t]^{1/2}f,\Qop_t^*[-\Delta_t]^{1/2}g}_{\Omega(t)}
   = \ip{[-\Delta_t]^{1/2}\Qop_t\Qop_t^*[-\Delta_t]^{1/2}f,g}_{\Omega(t)}
   = \ip{f,g}_{\Omega(t)}
\]
by what we have already obtained.
Hence $\Qop_t^*[-\Delta_t]^{1/2}$ defines an isometry from
$L^2(\Omega(t))$ into itself, and its adjoint
$[-\Delta_t]^{1/2}\Qop_t$ is a contraction with dense range in
$L^2(\Omega(t))$. To show that both operators are unitary,
it suffices to verify that $[-\Delta_t]^{1/2}\Qop_t$ is
injective.

So suppose that $[-\Delta_t]^{1/2}\Qop_t f = 0$ for some
$f \in L^2(\Omega(t))$. Since $\Qop_t f$ belongs to $H^1_0(\Omega(t))$,
this implies that $\Qop_t f = 0$. Then if $0 < \tau < t$
and we replace $\Qop_t f$ with the function having the same values on 
$\partial\Omega(\tau)$ and being harmonic in $\Omega(\tau)$, we still
have the zero function. Thus, by Lemma~\ref{le:Qop}, we must have
$\Qop_\tau f=0$ for all $\tau$ with $0 < \tau < t$.
Then the derivative must also vanish (for a.e.\ $\tau$ with
$0 < \tau < t$):
\[
   \frac{\partial}{\partial\tau}\Qop_\tau f(z) =
   \int_{\partial \Omega(\tau)}
   P_\tau(z,\zeta)f(\zeta)\varrho(\zeta)\,\ds(\zeta) = 0,
   \qquad z\in\Omega(\tau).
\] 
But this is the harmonic extension of $f\varrho$, and it can
vanish everywhere only if $f\varrho=0$ holds a.e.\ on
$\partial\Omega(\tau)$.  
Since $\varrho > 0$ a.e., we conclude that $f=0$ a.e.
This proves that $[-\Delta_t]^{1/2}\Qop_t$ is injective.
\end{proof}

\section{Basic random fields}
\label{sec:Random}

We briefly recall the construction and some basic properties of the
white noise and Gaussian free fields. We primarily adopt the
viewpoint of Gaussian Hilbert spaces. See Janson's book \cite{Janson} for an
account of such spaces and Sheffield~\cite{Sh} for a general survey of
the Gaussian free field.

Throughout this section we assume that $\Omega$
is a bounded domain in the complex plane with a $C^2$-smooth
boundary.

\subsection{The white noise field}
The (real-valued) white noise field on $\Omega$ is given formally as
a random linear combination
\[
   \Phi = \sum_{j=1}^{+\infty}\xi_j\phi_j,
\]
where the $\xi_j$ ($j=1,2,\ldots$) are all independent standard Gaussian random
variables, and the functions $\phi_j$ ($j=1,2,\ldots$) form an orthonormal 
basis of $L^2(\Omega)$. At times, we may want to express
this as $\Phi\in \mathrm{WN}(\Omega)$.
The above series does not converge (almost surely) in $L^2(\Omega)$.
However, in terms of the bilinear form
$\ip{\cdot,\cdot}_\Omega$ (i.e., the $L^2$ inner product on $\Omega$),
it makes sense to consider
\[
   \ip{f,\Phi}_\Omega
   = \sum_{j=1}^{+\infty}\xi_j\ip{f,\phi_j}_\Omega
\]
since this series converges in mean square 
to a Gaussian variable with mean $0$ and variance  
$\norm{f}_{L^2(\Omega)}^2$ for each $f \in L^2(\Omega)$. Thus
we may view the white noise field as a \emph{Gaussian Hilbert space}
formed by the jointly Gaussian variables $\ip{f,\Phi}_\Omega$ -- one
for each $f \in L^2(\Omega)$ -- whose covariance structure is given by
\[
   \Ex \bigl[ \ip{f,\Phi}_\Omega \ip{g,\Phi}_\Omega \bigr]
   = \ip{f,g}_\Omega.
\]
Here, the symbol $\Ex$ stands for the expectation operation. 
Note, in particular, that the law for
the field is independent of the choice of basis -- this is somewhat
akin to the fact that the reproducing kernel of a Hilbert space of
functions, which may be written in terms of an orthonormal basis, is
independent of the choice of the basis (cf.\ \cite{Aron}).

Alternatively, $\Phi$ can be understood as a random element in a space
bigger than $L^2(\Omega)$. One way is to define $\Phi$ as
a random distribution in the Sobolev space $H^{-1-\epsilon}(\Omega)$
for some $\epsilon > 0$ such that its action on test
functions $f \in C_0^\infty(\Omega)$ yields random variables
$\ip{f,\Phi}_\Omega$ as described above. This is possible because
the natural embedding of $L^2(\Omega)$ into $H^{-1-\epsilon}(\Omega)$
is a Hilbert--Schmidt operator; see \cite[Sec.~2.2]{Sh}.

An important property of the white noise field $\Phi$ is the
independent action of the ``vibrations'' in different parts of
$\Omega$. For instance, 
if $\Omega=\Omega_1\cup\Omega_2\cup E$,  where the union is
mutually disjoint, and $\Omega_j$ is open for $j=1,2$, while
$E$ has zero area, then 
\[
   L^2(\Omega)= L^2(\Omega_1)\oplus L^2(\Omega_2)
\]
as orthogonal subspaces in a Hilbert space. As a result, 
$\Phi=\Phi_1+\Phi_2=\Phi_1\boxplus\Phi_2$, where $\Phi_1$ is the
field $\Phi$ conditioned to vanish on $\Omega_2$, and $\Phi_2$ is
the field $\Phi$ conditioned to vanish on $\Omega_1$. Here, the
symbol $\boxplus$ is used to indicate that the summands are independent
random fields. In a natural sense, $\Phi_j\in\mathrm{WN}(\Omega_j)$
for $j=1,2$. 

\subsection{The Gaussian free field}
The Gaussian free field on $\Omega$
with vanishing boundary data, which we will denote as
$\Psi \in\mathrm{GFF}_0(\Omega)$, can be defined in a way that is
analogous to the white noise field above. Instead of the space
$L^2(\Omega)$, we just work with the Sobolev space $H^1_0(\Omega)$
(see Section \ref{sec:Sobolev}). Then we obtain a Gaussian Hilbert space
formed by centred Gaussian variables $\ip{f,\Psi}_{\Omega,\nabla}$
for $f \in H^1_0(\Omega)$ with the property that
\[
   \Ex \bigl[ \ip{f,\Psi}_{\Omega,\nabla}
            \ip{g,\Psi}_{\Omega,\nabla} \bigr]
   = \ip{f,g}_{\Omega,\nabla}.
\]
Again, $\Psi$ does not (almost surely)
determine an element of $H^1_0(\Omega)$, yet it can be defined as
a random element in the Sobolev space $H^{-\epsilon}(\Omega)$
for any $\epsilon > 0$. In view of \eqref{eq:H1ip}, the Gaussian free
field may be retrieved from the white noise by
$\Psi = [-\Delta]^{-1/2}\Phi$ in the sense of distribution theory.
Moreover, the action of $\Psi$ on $L^2(\Omega)$ is given by
$\ip{f,\Psi}_\Omega = \ip{[-\Delta]^{-1}f,\Psi}_{\Omega,\nabla}$,
so that
\begin{equation} \label{eq:GFFL2}
   \Ex \bigl[ \ip{f,\Psi}_\Omega \ip{g,\Psi}_\Omega \bigr]
   = \ip{f,[-\Delta]^{-1}g}_\Omega
\end{equation}
for $f,g \in L^2(\Omega)$.
In fact, the definition of $\ip{f,\Psi}_\Omega$ can be extended
to all $f \in H^{-1}(\Omega)$ since $[-\Delta]^{-1}$ maps
$H^{-1}(\Omega)$ into $H^1_0(\Omega)$.

A main difference between the fields $\mathrm{GFF}_0(\Omega)$ and
$\mathrm{WN}(\Omega)$ is that while $\mathrm{WN}(\Omega)$ is purely
local, in $\mathrm{GFF}_0(\Omega)$ we have non-trivial long-range
correlations, due to the non-local nature of the operator
$[-\Delta]^{-1/2}$.

\section{The Gaussian free field via the Hadamard operator}
\label{sec:GFF}

Throughout this section, we work in the setting of the assumptions
(a)--(d) described in Section~\ref{sec:HVar}. For brevity, we write
$\Omega = \Omega(1)$ for the largest domain under consideration.

\subsection{Construction}
We start by invoking the Hadamard operator to produce the Gaussian
free field from the white noise field.
We agree to write $\Psi_0 = 0$.  

\begin{thm} 
\label{thm:GFF}
Let $\Phi \in \mathrm{WN}(\Omega)$. For $0 < t \le 1$, let
$\Psi_t := \Qop_t\Phi$, i.e.\
\[
   \ip{f,\Psi_t}_{\Omega} := \ip{\Qop_t^*f,\Phi}_{\Omega},
   \qquad f \in H^{-1}(\Omega).
\] 
Then $\Psi_t \in \mathrm{GFF}_0(\Omega(t))$. Moreover, the
process $\Psi_t$ $(0 \leq t \leq 1)$ has independent increments:\
for all $0 = t(0) < \cdots < t(n) = 1$ and
$f_1,\ldots,f_n \in H^{-1}(\Omega)$, the random variables
$\ip{f_j,\Psi_{t(j)}-\Psi_{t(j-1)}}_{\Omega}$ $(j=1,\ldots,n)$
are independent.
\end{thm}

\begin{proof}
For $0 < t,t' \leq 1$ and $f,g \in H^{-1}(\Omega)$, we have from
the definitions that
\[
   \Ex\bigl[ \ip{f,\Psi_t}_{\Omega}
             \ip{g,\Psi_{t'}}_{\Omega} \bigr]
   = \Ex\bigl[ \ip{\Qop_t^*f,\Phi}_{\Omega}
               \ip{\Qop_{t'}^*g,\Phi}_{\Omega} \bigr]
   = \ip{\Qop_t^*f,\Qop_{t'}^*g}_{\Omega}
   = \ip{f,\Qop_t\Qop_{t'}^*g}_{\Omega}.
\]
According to Proposition~\ref{prop:uni},
$\Qop_t\Qop_{t'}^* = [-\Delta_{t\wedge t'}]^{-1}$. Taking $t=t'$ we
see that $\Psi_t$ satisfies the correlations condition \eqref{eq:GFFL2} 
on $\Omega(t)$.
Moreover, it follows easily that the increments of $\Psi_t$ are
uncorrelated and hence, being jointly Gaussian, independent.
\end{proof}

Since $[-\Delta_t]^{-1} = \Gop_t$ (see Sec.~\ref{sec:Green}), the
covariance structure of the process $\Psi_t$ can be written in the
form
\begin{equation} \label{eq:energy} \begin{split}
   \Ex \bigl[ \ip{f,\Psi_t}_\Omega \ip{g,\Psi_{t'}}_\Omega \bigr]
   &= \int_\Omega\int_\Omega G_{t\wedge t'}(z,w)f(z)g(w)\,\dA(z)\dA(w) \\
   &= \int_\Omega \nabla\Gop_{t\wedge t'}f \cdot
                \nabla\Gop_{t\wedge t'}g\,\dA.
\end{split} \end{equation}
In particular, the variance of $\ip{f,\Psi_t}_\Omega$ equals the
energy integral $\int_\Omega \abs{\nabla\Gop_t f}^2\,\dA$.

In view of the definition of the Hadamard operator, we see from
Theorem~\ref{thm:GFF} that
the Gaussian free field $\mathrm{GFF}_0(\Omega(t))$ can be thought of
as obtained by integrating up (in terms of the area integral) the
harmonic fields induced by Poisson extensions of point oscillations at
each point $\zeta \in \partial\Omega(\tau)$ with $0 < \tau < t$.
Moreover, for $0 < t < t' \leq 1$ the $\mathrm{GFF}_0$ field on
$\Omega(t')$ can be decomposed into independent fields:
\begin{equation}
   \Psi_{t'} = \Psi_t \boxplus (\Psi_{t'} - \Psi_t)
   = \Qop_t \Phi \boxplus (\Qop_{t'}-\Qop_t)\Phi.
\label{eq-sum:Markov}
\end{equation}
Here $\Psi_t$ is a $\mathrm{GFF}_0$ on the smaller domain
$\Omega(t)$ with zero continuation to
$\Omega(t') \sm \cl[\Omega(t)]$, and
$\Psi_{t'}-\Psi_t$ is a field which is harmonic in
$\Omega(t)$ and coincides with $\Psi_{t'}$ on
$\Omega(t') \sm \cl[\Omega(t)]$ (see Lemma~\ref{le:Qop}).
The decomposition \eqref{eq-sum:Markov} expresses the 
Markov property of the Gaussian free field (cf.~\cite{Sh}).

\subsection{Time-derivative of $\Psi_t$}

Looking the increments at the infinitesimal scale, we may examine the
time-derivative of the process $\Psi_t$. To this end, it is convenient 
to use the \emph{harmonic sweep} operator $\Pop_t^*$ given by
\[
   \Pop_t^*f(\zeta) := \int_{\Omega(t)} P_t(z,\zeta)f(z)\,\dA(z),
   \qquad \zeta \in \partial\Omega(t).
\]
It is the adjoint of the Poisson extension operator. We see from the
definition of $\Qop_t^*$ that $\Qop_t^*f(\zeta) =
\Pop_\tau^*f(\zeta)$ provided that $\zeta \in \partial\Omega(\tau)$ with
$0 < \tau < t$. Consequently, for all $0 < t,t' \leq 1$ and
$f,g \in L^2(\Omega)$, we have, in view of \eqref{eq:polcoord}, that
\[
   \Ex\bigl[ \ip{f,\Psi_t}_{\Omega}
             \ip{g,\Psi_{t'}}_{\Omega} \bigr]
   = \ip{\Qop_t^*f,\Qop_{t'}^*g}_{\Omega}
   = \int_0^{t\wedge t'} \biggl\{ \int_{\partial\Omega(\tau)}
     \Pop_\tau^*f\Pop_\tau^*g\,\varrho\,\ds \biggr\} \,\diff\tau.
\]
Writing $\dottt{\Psi}_t$ for the time-derivative of $\Psi_t$ in the sense
of distribution theory, we 
get the expression
\begin{equation}
   \Ex\bigl[ \ip{f,\dottt{\Psi}_t}_{\Omega}
             \ip{g,\dottt{\Psi}_{t'}}_{\Omega} \bigr]
   = \delta_0(t-t')
     \int_{\partial\Omega(t)} \Pop_t^*f\,\Pop_t^*g\,\varrho\,\ds.
\label{eq-5corr}
\end{equation}
Next, suppose that $\Xi_t$ is a weighted white noise process on
$\partial\Omega(t)$ with weight $\varrho^{1/2}$, i.e.,
a Gaussian random field acting on $L^2(\partial\Omega(t))$ with 
correlation structure
\begin{equation}
    \Ex\bigl[ \ip{\phi,\Xi_t}_{\partial\Omega(t)}
       \ip{\psi,\Xi_t}_{\partial\Omega(t)} \bigr]
       = \int_{\partial\Omega(t)} \phi\psi \varrho \,\ds
\label{eq-5corr2}
\end{equation}
for $\phi,\psi \in L^2(\partial\Omega(t))$. Let $\Pop_t\Xi_t$
be the Poisson extension of $\Xi_t$, i.e., the harmonic field on
$\Omega(t)$ given by $\ip{f,\Pop_t\Xi_t}_{\Omega(t)} =
\ip{\Pop_t^*f,\Xi_t}_{\partial\Omega(t)}$ for $f \in L^2(\Omega(t))$;
in view of \eqref{eq-5corr2}, it has the correlation structure
\begin{equation}
    \Ex\bigl[ \ip{f,\Pop_t\Xi_t}_{\Omega(t)}
       \ip{g,\Pop_t\Xi_t}_{\Omega(t)} \bigr]
       = \int_{\partial\Omega(t)} \Pop^*_tf\,\Pop^*_tg\, \varrho \,\ds,
\label{eq-5corr3}
\end{equation}
which we compare with \eqref{eq-5corr}.
We arrive at the following conclusions, whose detailed verification is 
left to the reader:

\begin{prop} \label{prop:GFFder}
Suppose that $\Psi_t$ $(0 < t \leq 1)$ are given by
Theorem~\ref{thm:GFF}. Then the following hold:

\smallskip

(\textit{a})
For a fixed $f$, we have
\[
   \ip{f,\Psi_t}_{\Omega}
   \overset{d}{=} \int_0^t \sqrt{\kappa(\tau)} \,dB(\tau)
   \overset{d}{=} B\Bigl(\int_0^t \kappa(\tau)\,d\tau\Bigr),
   \qquad 0 \leq t \leq 1,
\]
where $B$ is a standard Brownian motion with $B(0) = 0$ and
$\kappa(\tau) := \int_{\partial\Omega(\tau)}
\abs{\Pop_\tau^*f}^2\varrho\,\ds$. Here, ``$\overset{d}{=}$'' indicates
that the processes have the same law (i.e.,\ are versions of each other).
\smallskip

(\textit{b})
The covariance structure of the differentiated process 
$\dottt{\Psi}_t$ is induced by the process $\Pop_t\Xi_t$:
\[
\Ex\big[\ip{f,\dottt{\Psi}_t}_{\Omega}\ip{g,\dottt{\Psi}_{t'}}_{\Omega}\big]
=\delta_0(t-t')\,\Ex\big[\ip{f,\Pop_t\Xi_t}_{\Omega(t)}\ip{g,\Pop_t\Xi_t}
_{\Omega(t)}\big].
\]
This identity should be understood in the sense of distribution theory.
\end{prop}

Thus, in a sense, we are dealing with a time-changed Brownian motion in
the time variable $t$. On the other hand, the individual independent
increments (at infinitesimal scale) correspond to the harmonic extension to 
$\Omega(t)$ of weighted white noise fields along the boundaries 
$\partial\Omega(t)$.

\begin{rem}
The independent decomposition of the Gaussian free field given by
Proposition~\ref{thm:GFF}
-- and the continuous limit of that decomposition alluded to above --
are in many ways analogous to the Brownian exploration of the Gaussian
free along a space-filling curve, as in \cite{Sh}.
\end{rem}

\subsection{Boundary averages}

It is rather well known (see e.g.\ \cite{Sh,DS}) that the averages
of the Gaussian free field on concentric circles follow a time-changed
Brownian motion. We want to formulate an analogue of this fact for the
averages of $\Psi_1$ on the associated flow of the boundaries
$\partial\Omega(t)$.

We first note that Lemma~\ref{le:Qop} implies that for
$f \in L^2(\Omega)$ and $0 < t < 1$, we have
$\Pop_t \Qop_1 f = (\Qop_1 - \Qop_t)f$ on $\Omega(t)$.
Passing to the adjoints, we get
$\Qop_1^* \Pop_t^* f = (\Qop_1^* - \Qop_t^*)f$ for functions $f$
on $\Omega(t)$. Thus we may evaluate $\Psi_1$ with the
distribution $\Pop_t^* f \,\ds$ as a ``test function'' to obtain
\[
   \int_{\partial\Omega(t)} \Psi_1 \Pop_t^*f \,\ds
   = \ip{\Pop_t^*f, \Psi_1}_{\partial\Omega(t)}
   = \ip{\Qop_1^*\Pop_t^* f, \Phi}_{\Omega}
   = \ip{(\Qop_1^*-\Qop_t^*)f,\Phi}_{\Omega}
   = \ip{f,\Psi_1-\Psi_t}_{\Omega}.
\]
More generally, this reasoning can be extended to the case where $f$
is replaced by a distribution whose support is contained in
$\Omega(t)$. Then the following result is obtained.

\begin{cor}
Let $X_t(f) := \int_{\partial\Omega(t)} \Psi_1 \Pop_t^*f\,\ds$
for $0 < t < 1$ and distributions $f$ supported on the skeleton
$\bigcap_t\Omega(t)$. Then, for all such distributions $f$ and $g$,
and $0 < t,t' < 1$, we have
\[
   \Ex \bigl[X_t(f)X_{t'}(g)\bigr]
   = \int_{t\vee t'}^1 \biggl\{ \int_{\partial\Omega(\tau)}
     \Pop_\tau^*f \,\Pop_\tau^*g \,\varrho\,\ds \biggr\} \,\diff\tau,
\]
where $t \vee t'$ is the maximum of $t$ and~$t'$. In particular,
$X_t(f) \overset{d}{=} B\bigl(\int_t^1 \kappa(\tau)\,d\tau\bigr)$
with $\kappa(\tau)$ defined as in Proposition~\ref{prop:GFFder}.
\label{cor-5.5}
\end{cor}

Thus, the process $X_t(f)$ ($0<t<1$) -- the values of which are the
integrals along the curves $\partial\Omega(t)$ of the field 
$\Psi_1\in \mathrm{GFF}_0(\Omega)$ with respect to the (signed) measures
$\Pop_t^*f\,\ds$ -- obeys the law of a time-reversed Brownian motion
running at a variable speed.

For instance, if $\Omega(t)$ is the open disc of radius $t$ centred
at the origin and $f = \delta_0$, the unit point mass at $0$, we have 
$\varrho \equiv 1$ and $\Pop_t^*\delta_0 \equiv (2\pi t)^{-1}$. Then
$\kappa(t) = (2\pi t)^{-1}$, and we obtain the well-known
result that the average of $\mathrm{GFF}_0$ on concentric
circles of radii $e^{-2\pi t}$ has exactly the
law of a Brownian motion.

\begin{rem}
It is possible to supply an expression of type \eqref{eq:energy} for the
covariance $\Ex[X_t(f)X_{t'}(g)]$ altough it is slightly more delicate.
We first observe that the Green potential $\Gop_1 f$ is well defined
(as a distribution) for distributions $f$ supported in the skeleton, and
that $\Gop_1 f$ is harmonic in $\Omega$ off the skeleton and vanishes on 
$\partial\Omega$. We let $\tilde{\Gop}_t f$ denote the function which equals
$\Gop_1 f$ on $\Omega\sm\Omega(t)$, and is harmonically extended
from its boundary values to $\Omega(t)$. Then
\[
  \Ex\bigl[X_t(f)X_{t'}(g)\bigr]
  = \int_\Omega \nabla\tilde{\Gop}_{t\vee t'}f \cdot
    \nabla\tilde{\Gop}_{t\vee t'}g \,\dA.
\]
\end{rem}

\begin{rem}
(a) As we have seen, the $\mathrm{GFF}_0$ field can be derived from
the white noise field by applying the inverse of the square root of
minus the Laplacian to $\mathrm{WN}$. The conformal invariance of the
Laplacian leads to the corresponding conformal invariance of
$\mathrm{GFF}_0$. If we replace the Laplacian by
$\operatorname{div}\frac{1}{\omega}\nabla$ (which appears, e.g., in
Calder\'on's inverse conductivity problem), the process retains the
conformal invariance structure if we allow the weight $\omega$ to
change correspondingly. Here, the weight function $\omega$ may be
assumed positive and smooth. It should be possible to extend the use
the Hadamard variation technique to this more general setting. 
We might also want to consider in the complex-valued setting the
operators introduced by Garabedian~\cite{Gar}. 

(b) Hadamard's variation formula also applies in the setting of the 
biharmonic equation with vanishing Dirichlet data (see \cite{H} for
the original result, and \cite{Hed,HJS} for the integral version).
This should lead to a corresponding Hadamard operator in this setting,
and that should have applications to the associated random fields.  
\end{rem}

\subsection*{Acknowledgements}
The authors thank Michael Benedicks and Kalle Kyt\"ol\"a for their
interest in this work, and Boualem Djehiche for help with references.
The first author also thanks Stanislav Smirnov for the inspiring conference
``Conformal maps from probability to physics'' at Monte Verit\`a,
Ascona, in 2010, and Pavel Wiegmann for several interesting
conversations. The second author acknowledges the hospitality of
the Department of Mathematics at the Royal Institute of Technology
in Stockholm during his visit in 2011.

\end{document}